\documentclass[11pt, english]{amsart}%
\usepackage{latexsym}
\usepackage{amscd}
\usepackage{amssymb}
\usepackage[all, cmtip]{xy}
\usepackage{amsmath}
\usepackage{graphics, graphicx, calc, tikz, tkz-graph, bm, babel, ragged2e,
enumerate, etoolbox, xcolor, caption, url}
\usepackage{amsfonts}
\usepackage{graphicx}%
\providecommand{\U}[1]{\protect\rule{.1in}{.1in}}
\providecommand{\U}[1]{\protect\rule{.1in}{.1in}}
\providecommand{\U}[1]{\protect\rule{.1in}{.1in}}
\usetikzlibrary{decorations.markings, arrows}
\DeclareMathAlphabet{\eusm}{OT1}{eusm}{m}{n}
\newtheorem{theorem}{Theorem}[section]
\newtheorem{prop}[theorem]{Proposition}

\newtheorem{example}[theorem]{Example}
\newtheorem{remark}[theorem]{Remark}
\newtheorem{question}[theorem]{Question}

\usetikzlibrary{decorations.markings}
\tikzstyle{vertex}=[circle, draw, fill=black, inner sep=0pt, minimum size=6pt]

}}}]}

\begin{document}
\title[Leavitt path algebras with bounded index of nilpotence]{Leavitt path algebras with bounded index of nilpotence}
\author{Kulumani M. Rangaswamy}
\address{Departament of Mathematics, University of Colorado at Colorado Springs,
Colorado-80918, USA}
\email{krangasw@uccs.edu}
\author{Ashish K. Srivastava}
\address{Department of Mathematics and Statistics, St. Louis University, St. Louis,
MO-63103, USA}
\email{ashish.srivastava@slu.edu}
\thanks{The work of the second author is partially supported by a grant from Simons
Foundation (grant number 426367).}
\maketitle

\begin{abstract}
In this paper we completely describe graphically Leavitt path algebras with
bounded index of nilpotence. We show that the Leavitt path algebra $L_{K}(E)$
has index of nilpotence at most $n$ if and only if no cycle in the graph $E$
has an exit and there is a fixed positive integer $n$ such that the number of
distinct paths that end at any given vertex $v$ (including $v$, but not
including the entire cycle $c$ in case $v$ lies on $c$) is less than or equal
to $n$. Interestingly, the Leavitt path algebras having bounded index of
nilpotence turn out to be precisely those that satisfy a polynomial identity.
Furthermore, Leavitt path algebras with bounded index of nilpotence are shown
to be directly-finite and to be $\mathbb{Z}$-graded $\Sigma$-$V$ rings. As an
application of our results, we answer an open question raised in \cite{JST}
whether an exchange $\Sigma$-$V$ ring has bounded index of nilpotence.

\end{abstract}

\bigskip

\bigskip

\section{Introduction}

\noindent The objective of this paper is to characterize Leavitt path algebras
with bounded index of nilpotence. A ring $R$ is said to have \textit{bounded
index of} \textit{nilpotence} if there is a positive integer $n$ such that
$x^{n}=0$ for all nilpotent elements $x$ in $R$. If $n$ is the least such
integer then $R$ is said to have \textit{index of nilpotence $n$}. We show
that the Leavitt path algebra $L:=L_{K}(E)$ of a directed graph $E$ over a
field $K$ has index of nilpotence at most $n$ if and only if no cycle in the
graph $E$ has an exit and there is a fixed positive integer $n$ such that the
number of distinct paths that end at any given vertex $v$ (including $v$, but
not including the cycle $c$ in case $v$ lies on $c$) is less than or equal to
$n$. In this case, $L$ becomes a subdirect product of matrix rings $M_{t}(K)$
or $M_{t}(K[x,x^{-1}])$ of finite order $t\leq n$. Examples are constructed
showing that $L$ need not decompose as a direct sum of these matrix rings
$M_{t}(K)$ or $M_{t}(K[x,x^{-1}])$, though the decomposition is possible when
$E$ is row-finite. We show that a Leavitt path algebra $L$ with bounded index
of nilpotence is always directly-finite and that $L$ is a $%
\mathbb{Z}
$-graded $\Sigma$-$V$ ring, that is, each graded simple left/right $L$-module
is graded $\Sigma$-injective. Examples show that the converse of these
statements do not hold. Interestingly, it turns out that the graphical
conditions on $E$ that ensure $L$ has a bounded index of nilpotence are
exactly the same graphical conditions on $E$ that were shown in \cite{BLR} to
imply that $L$ satisfies a polynomial identity. When $E$ is a finite graph,
these graphical conditions also imply that $L$ has GK-dimension $\leq1$. Such
statements illustrate a unique phenomenon in the study of Leavitt path
algebras where a single graph property of $E$ often implies different
ring-theoretic properties for $L$ and these ring-theoretic properties are
usually independent of each other for general rings (see \cite{R1} for several
illustrations of this phenomenon of Leavitt path algebras). This feature of
Leavitt path algebras makes them really useful tools in constructing examples
of rings of various desired ring-theoretic properties. Finally, as an
application of our results, we answer a question raised in \cite{JST}, whether
an exchange $\Sigma$-$V$ ring has bounded index of nilpotence.

For the general notation, terminology and results in Leavitt path algebras, we
refer the reader to \cite{AArS}.\ We give below an outline of some of the
needed basic concepts and results.

A (directed) graph $E=(E^{0},E^{1},r,s)$ consists of two sets $E^{0}$ and
$E^{1}$ together with maps $r,s:E^{1}\rightarrow E^{0}$. The elements of
$E^{0}$ are called \textit{vertices} and the elements of $E^{1}$
\textit{edges}.

A vertex $v$ is called a \textit{sink} if it emits no edges and a vertex $v$
is called a \textit{regular} \textit{vertex} if it emits a non-empty finite
set of edges. An \textit{infinite emitter} is a vertex which emits infinitely
many edges. For each $e\in E^{1}$, we call $e^{\ast}$ a ghost edge. We let
$r(e^{\ast})$ denote $s(e)$, and we let $s(e^{\ast})$ denote $r(e)$.
A\textit{\ path} $\mu$ of length $n>0$ is a finite sequence of edges
$\mu=e_{1}e_{2}\cdot\cdot\cdot e_{n}$ with $r(e_{i})=s(e_{i+1})$ for all
$i=1,\cdot\cdot\cdot,n-1$. In this case $\mu^{\ast}=e_{n}^{\ast}\cdot
\cdot\cdot e_{2}^{\ast}e_{1}^{\ast}$ is the corresponding ghost path. A vertex
is considered a path of length $0$.

A path $\mu$ $=e_{1}\dots e_{n}$ in $E$ is \textit{closed} if $r(e_{n}%
)=s(e_{1})$, in which case $\mu$ is said to be \textit{based at the vertex
}$s(e_{1})$. A closed path $\mu$ as above is called \textit{simple} provided
it does not pass through its base more than once, i.e., $s(e_{i})\neq
s(e_{1})$ for all $i=2,...,n$. The closed path $\mu$ is called a
\textit{cycle} if it does not pass through any of its vertices twice, that is,
if $s(e_{i})\neq s(e_{j})$ for every $i\neq j$. An \textit{exit }for a path
$\mu=e_{1}\dots e_{n}$ is an edge $e$ such that $s(e)=s(e_{i})$ for some $i$
and $e\neq e_{i}$.

An \textit{infinite rational path} $p$ is an infinite path of the form
$p=x_{1}x_{2} \ldots x_{n} \ldots$ where there is an $m \geq1$ such that
$x_{k} = g$, a fixed closed path for all $k \geq m$ and that $x_{k}$ is an
edge $e_{k}$ if $k <m$. Thus $p$ will be of the form $p=e_{1}e_{2}\cdot
\cdot\cdot e_{m-1}gggg\cdot\cdot\cdot$ where $g$ is a closed path and the
$e_{i}$ are edges. An infinite path which is not rational is called an
\textit{irrational path}.

A graph $E$ is said to satisfy \textit{Condition (K)}, if every vertex $v$ on
a closed path $c$ is also the base of a another closed path $c^{\prime}%
$different from $c$. A graph $E$ is said to satisfy \textit{Condition (L)} if
every cycle has an exit.

If there is a path from vertex $u$ to a vertex $v$, we write $u\geq v$. A
subset $D$ of vertices is said to be \textit{downward directed }\ if for any
$u,v\in D$, there exists a $w\in D$ such that $u\geq w$ and $v\geq w$. A
subset $H$ of $E^{0}$ is called \textit{hereditary} if, whenever $v\in H$ and
$w\in E^{0}$ satisfy $v\geq w$, then $w\in H$. A hereditary set is
\textit{saturated} if, for any regular vertex $v$, $r(s^{-1}(v))\subseteq H$
implies $v\in H$.

Given an arbitrary graph $E$ and a field $K$, the \textit{Leavitt path algebra
}$L_{K}(E)$ is defined to be the $K$-algebra generated by a set $\{v:v\in
E^{0}\}$ of pair-wise orthogonal idempotents together with a set of variables
$\{e,e^{\ast}:e\in E^{1}\}$ which satisfy the following conditions:

(1) \ $s(e)e=e=er(e)$ for all $e\in E^{1}$.

(2) $r(e)e^{\ast}=e^{\ast}=e^{\ast}s(e)$\ for all $e\in E^{1}$.

(3) (The ``CK-1 relations") For all $e,f\in E^{1}$, $e^{\ast}e=r(e)$ and
$e^{\ast}f=0$ if $e\neq f$.

(4) (The ``CK-2 relations") For every regular vertex $v\in E^{0}$,
\[
v=\sum_{e\in E^{1},s(e)=v}ee^{\ast}.
\]
Every Leavitt path algebra $L_{K}(E)$ is a $\mathbb{Z}$\textit{-graded
algebra}, namely, $L_{K}(E)={\displaystyle\bigoplus\limits_{n\in\mathbb{Z}}}
L_{n}$ induced by defining, for all $v\in E^{0}$ and $e\in E^{1}$, $\deg
(v)=0$, $\deg(e)=1$, $\deg(e^{\ast})=-1$. Here the $L_{n}$ are abelian
subgroups satisfying $L_{m}L_{n}\subseteq L_{m+n}$ for all $m,n\in\mathbb{Z}
$. Further, for each $n\in\mathbb{Z}$, the \textit{homogeneous component
}$L_{n}$ is given by
\[
L_{n}=\{ {\textstyle\sum} k_{i}\alpha_{i}\beta_{i}^{\ast}\in L:\text{ }%
|\alpha_{i}|-|\beta_{i}|=n\}.
\]
Elements of $L_{n}$ are called \textit{homogeneous elements}. An ideal $I$ of
$L_{K}(E)$ is said to be a \textit{graded ideal} if $I=$
${\displaystyle\bigoplus\limits_{n\in\mathbb{Z}}} (I\cap L_{n})$. If $A,B$ are
graded modules over a graded ring $R$, we write $A\cong_{gr}B$ if $A$ and $B$
are graded isomorphic and we write $A\oplus_{gr}B$ to denote a graded direct
sum. We will also be using the usual grading of a matrix of finite order. For
this and for the various properties of graded rings and graded modules, we
refer to \cite{H-1}, \cite{HR} and \cite{NV}.

A \textit{breaking vertex }of a hereditary saturated subset $H$ is an infinite
emitter $w\in E^{0}\backslash H$ with the property that $0<|s^{-1}(w)\cap
r^{-1}(E^{0}\backslash H)|<\infty$. The set of all breaking vertices of $H$ is
denoted by $B_{H}$. For any $v\in B_{H}$, $v^{H}$ denotes the element
$v-\sum_{s(e)=v,r(e)\notin H}ee^{\ast}$. Given a hereditary saturated subset
$H$ and a subset $S\subseteq B_{H}$, $(H,S)$ is called an \textit{admissible
pair.} Given an admissible pair $(H,S)$, the ideal generated by $H\cup
\{v^{H}:v\in S\}$ is denoted by $I(H,S)$. It was shown in \cite{T} that the
graded ideals of $L_{K}(E)$ are precisely the ideals of the form $I(H,S)$ for
some admissible pair $(H,S)$. Moreover, $L_{K}(E)/I(H,S)\cong L_{K}%
(E\backslash(H,S))$. Here $E\backslash(H,S)$ is a \textit{Quotient graph of
}$E$ where $(E\backslash(H,S))^{0}=(E^{0}\backslash H)\cup\{v^{\prime}:v\in
B_{H}\backslash S\}$ and $(E\backslash(H,S))^{1}=\{e\in E^{1}:r(e)\notin
H\}\cup\{e^{\prime}:e\in E^{1} $ with $r(e)\in B_{H}\backslash S\}$ and $r,s$
are extended to $(E\backslash(H,S))^{0}$ by setting $s(e^{\prime})=s(e)$ and
$r(e^{\prime})=r(e)^{\prime}$. It is known (see \cite{R}) that if $P$ is a
prime ideal of $L$ with $P\cap E^{0}=H$, then $E^{0}\backslash H$ is downward directed.

Let $\Lambda$ be an infinite set and $R$, a ring. Then $M_{\Lambda}(R)$
denotes the ring of $\Lambda\times\Lambda$ matrices in which all except at
most finitely many entries are non-zero.

\bigskip

\section{Results}

\noindent In this section, we characterize Leavitt path algebras having
bounded index of nilpotence. We begin with the following useful proposition
some part of which might be implicit in earlier works on Leavitt path algebras.

\begin{prop}
\label{matrixCreation} Let $E$ be an arbitrary graph and let $L:=L_{K}(E)$.

\begin{enumerate}
[(a)]

\item Let $v$ be a vertex in $E$ which does not lie on a closed path. If, for
some $n\geq1$, there are $n$ distinct paths $p_{1},\cdot\cdot\cdot,p_{n}$ in
$E$ that end at $v$, then the set $T_{n}=\{{\displaystyle\sum\limits_{i=1}%
^{n}}{\displaystyle\sum\limits_{j=1}^{n}}k_{ij}p_{i}p_{j}^{\ast}:k_{ij}\in
K\}$ is a subring of $L$ isomorphic to the matrix ring $M_{n}(K)$.

\item Let $v$ be a vertex in $E$ lying on a cycle $c$ and let $f$ be an exit
for $c$ at $v$. Then, for every integer $n\geq1$, the subset

$S_{n}=\{{\displaystyle\sum\limits_{i=1}^{n}}$ ${\displaystyle\sum
\limits_{j=1}^{n}}k_{ij}c^{i}ff^{\ast}(c^{\ast})^{j}:k_{ij}\in K\}$ is a
subring of $L$ isomorphic to the matrix ring $M_{n}(K)$.

\item Let $v$ be a vertex lying on a cycle $c$ without exits in $E$. If, for
some $n\geq1$, there are $n$ distinct paths $p_{1},\cdot\cdot\cdot,p_{n}$ in
$E$ that end at $v$ and do not go through the entire cycle $c$, then again the
set $T_{n}=\{{\displaystyle\sum\limits_{i=1}^{n}}{\displaystyle\sum
\limits_{j=1}^{n}}k_{ij}p_{i}p_{j}^{\ast}:k_{ij}\in K\}$ is a subring of $L$
isomorphic to the matrix ring $M_{n}(K)$.
\end{enumerate}
\end{prop}

\begin{proof}
(a) First observe that $p_{j}^{\ast}p_{k}\neq0$ if and only if $p_{j}=p_{k}$.
Because, if $p_{j}^{\ast}p_{k}\neq0$, then either $p_{j}=p_{k}p^{\prime}$ or
$p_{k}=p_{j}q^{\prime}$ for some paths $p^{\prime},q^{\prime}$. Since
$r(p_{j})=r(p_{k})=v$, we get $s(p^{\prime})=v=r(p^{\prime})$ and
$s(q^{\prime})=v=r(q^{\prime})$. Since $v$ does not lie on a closed path, we
conclude that $p^{\prime}=v=q^{\prime}$. So $p_{j}=p_{k}$. Conversely, if
$p_{j}=p_{k}$, then clearly $p_{j}^{\ast}p_{k}=$ $p_{j}^{\ast}p_{j}=v\neq0$.
For all $i,j$, let $\varepsilon_{ij}=p_{i}p_{j}^{\ast}$. Clearly,
$(\varepsilon_{ii})^{2}=\varepsilon_{ii}$ and $\varepsilon_{ij}\varepsilon
_{kl}=\varepsilon_{il}$ or $0$ according as $j=k$ or not. Thus the
$\varepsilon_{ij}$ form a set of matrix units and it is readily seen that the
set  $T_{n}=\{{\displaystyle\sum\limits_{i=1}^{n}}{\displaystyle\sum
\limits_{j=1}^{n}}k_{ij}p_{i}p_{j}^{\ast}:k_{ij}\in K\}$ is a subring of $L$
isomorphic to the matrix ring $M_{n}(K)$.

(b) Suppose $c$ is a cycle in $E$ with an exit $f$ at a vertex $v$. Consider
the set $\{\varepsilon_{ij}=c^{i}ff^{\ast}(c^{\ast})^{j}:1\leq i,j\leq n\}$.
Clearly, the $\varepsilon_{ij}$ form a set of matrix units as $(\varepsilon
_{ii})^{2}=\varepsilon_{ii}$ and $\varepsilon_{ij}\varepsilon_{kl}%
=\varepsilon_{il}$ or $0$ according as $j=k$ or not. It is then easy to check
that the set $S_{n}=\{{\displaystyle\sum\limits_{i=1}^{n}}$
${\displaystyle\sum\limits_{j=1}^{n}}k_{ij}c^{i}ff^{\ast}(c^{\ast})^{j}%
:k_{ij}\in K\}$ is a subring of $L$ isomorphic to the matrix ring $M_{n}(K)$.

(c) Let $v$ be the base of a cycle $c$ without exits and $p_{1},\cdot
\cdot\cdot,p_{n}$ be $n$ distinct paths that end at $v$ and not go through the
entire cycle $c$. Using the fact that $c$ is a cycle without exits and
repeating the arguments as in (a), it follows that $p_{j}^{\ast}p_{k}\neq0$ if
and only if $p_{j}=p_{k}$. As before let $\varepsilon_{ij}=p_{i}p_{j}^{\ast}$
with $1\leq i,j\leq n$. Clearly, $(\varepsilon_{ii})^{2}=\varepsilon_{ii}$ and
$\varepsilon_{ij}\varepsilon_{kl}=\varepsilon_{il}$ or $0$ according as $j=k$
or not. Thus the $\varepsilon_{ij}$ form a set of matrix units and it is
readily seen that $T_{n}=\{{\displaystyle\sum\limits_{i=1}^{n}}%
{\displaystyle\sum\limits_{j=1}^{n}}k_{ij}p_{i}p_{j}^{\ast}:k_{ij}\in K\}$ is
a subring of $L$ isomorphic to the matrix ring $M_{n}(K)$.
\end{proof}

We are now ready to describe all the Leavitt path algebras with bounded index
of nilpotence. It is interesting to note that the Leavitt path algebras having
bounded index of nilpotence are precisely those that satisfy a polynomial identity.

Recall that an algebra $A$ over a field $K$ is said to satisfy a polynomial
identity if there exists a non-zero element $f$ in $K[x_{1},\ldots,x_{n}]$
such that $f(a_{1},\ldots,a_{n})=0$ for all $a_{i}$ in $A$. Clearly every
commutative ring satisfies a polynomial identity but there are many
interesting classes of noncommutative rings too that satisfy a polynomial
identity.\ For instance, the Amitsur-Levitzky theorem (see \cite{P}) states
that, for any $n\geq1$, the matrix ring $M_{n}(R)$ over a commutative ring $R$
satisfies a polynomial identity of degree $2n$. In \cite{BLR} it is shown
that the Leavitt path algebra $L_{K}(E)$ of an arbitrary graph $E$ over a
field $K$ satisfies a polynomial identity if and only if no cycle in $E$ has
an exit and there is a fixed positive integer $d$ such that the number of
distinct paths that end at any given vertex $v$ (including $v$, but not
including the entire cycle $c$ in case $v$ lies on $c$) is less than or equal
to $d$. When $E$ is a finite graph, then the Leavitt path algebra $L_{K}(E)$
satisfying a polynomial identity is known to be equivalent to the
Gelfand-Kirillov dimension of $L_{K}(E)$ being at most one \cite{BLR}.

\begin{theorem}
\label{bdd Index} Let $E$ be an arbitrary graph. Then the following properties
are equivalent for $L:=L_{K}(E)$:

\begin{enumerate}
\item $L$ has index of nilpotence less than or equal to $n$;

\item No cycle in $E$ has an exit and there is a fixed positive integer $n$
such that the number of distinct paths that end at any given vertex $v$
(including $v$, but not including the entire cycle $c$ in case $v$ lies on
$c$) is less than or equal to $n$;

\item For any graded prime ideal $P$ of $L$, $L/P\cong_{gr}M_{t}(K)$ or
$M_{t}(K[x,x^{-1}])$ where $t\leq n$ with appropriate matrix gradings;

\item $L$ is a graded subdirect product of graded rings $\{A_{i}:i\in I\}$
where, for each $i$, $A_{i}\cong_{gr}M_{t_{i}}(K)$ or $M_{t_{i}}(K[x,x^{-1}])
$ with appropriate matrix gradings where, for each $i$, $t_{i}\leq n$, a fixed
positive integer.

\item $L$ satisfies a polynomial identity.
\end{enumerate}
\end{theorem}

\begin{proof}
Assume (i), that is, assume that the index of nilpotence of $L$ is $\leq n$.
We claim that no cycle in $E$ can have an exit. Because, otherwise, by
Proposition \ref{matrixCreation} (b), $L$ will contain subrings of matrices of
arbitrary finite size and this will give rise to unbounded index of nilpotence
for $L$. Thus every vertex $v$ in $E$ either does not lie on a closed path or
lies on a cycle without exits. If there are more than $n$ distinct paths
ending at $v$, then again, by Proposition \ref{matrixCreation} (a) and (c),
$L$ will contain a copy of a matrix ring of order greater than $n$ over $K$
which will imply that the index of nilpotence of $L$ is greater than $n$, a
contradiction. This proves (ii).

Assume (ii). Let $P=I(H,S)$ be a graded prime ideal of $L$. Our hypothesis
implies that no cycle in $E\backslash(H,S)$ has an exit and that $n$ is also
the upper bound for the number of distinct paths ending at any vertex in
$E\backslash(H,S)$. So $E\backslash(H,S)$ contains no infinite irrational
paths. This means that every path ends at a sink or at a cycle without exits.
Also, as $I(H,S)$ is a graded prime ideal, Theorem 3.12 of \cite{R} implies
that $(E\backslash(H,S))^{0}$ is downward directed. Consequently,
$E\backslash(H,S)$ contains either (a) exactly one sink $w$ or (b) exactly one
cycle $c$ without exits based at a vertex $v$. Now in case (a), there are no
more than $n$ distinct paths ending at $w$ and, in  case (b), there are no
more than $n$ paths which end at $v$ and do not go through the cycle $c$. We
then appeal to Corollary 2.6.5 and Lemma 2.7.1 of \cite{AArS} to conclude that
$L/P\cong L_{K}(E\backslash(H,S))\cong M_{t}(K)$ or $M_{t}(K[x,x^{-1}])$
according as $E\backslash(H,S)$ contains a sink or a cycle without exits. This
proves (iii).

Assume (iii). Now, for any graded prime ideal $P$, $L/P\cong_{gr}M_{t}(K)$ or
$M_{t}(K[x,x^{-1}])$ with appropriate matrix gradings where $t\leq n$, a fixed
positive integer. It is known that the intersection $I$ of all graded prime
ideals of $L$ is zero. For the sake of completeness, we shall outline the
argument. If $I\neq0$, being a graded ideal, $I$ the will contain vertices.
But, given any vertex $v$, a graded ideal $M$ maximal among graded ideals with
respect to $v\notin M$ is a graded prime ideal, because, for any two
homogeneous elements $a,b$, if $a\notin M$ and $b\notin M$, then $v\in M+LaL$
and $v\in M+LbL$. Consequently, $v=v^{2}\in(M+LaL)(M+LbL)=M+LaLbL$. Since
$v\notin M$, $aLb\nsubseteqq M$. Thus $M$ is a graded prime ideal. But then
$v\notin M$ implies $v\notin I$, a contradiction. Thus $\cap\{P:P$ graded
prime ideal$\}=0$. Consequently, $L$ is a graded subdirect product of the
graded rings $L/P$ graded isomorphic to $M_{t}(K)$ or $M_{t}(K[x,x^{-1}])$
under appropriate matrix gradings, where $t\leq n$, a fixed positive integer.
This proves (iv).

Assume (iv). If $t\leq n$, then the matrix rings $M_{t}(K)$ and $M_{t}%
(K[x,x^{-1}])$ will each have index of nilpotence $\leq n$. Consequently, a
subdirect product of such rings will also have nilpotence index $\leq n$. This
proves (i).

Assume (iv). By the Amitsur-Levitzky theorem \cite{P}, both $M_{t}(K)$ and
$M_{t}(K[x,x^{-1}])$ with $t\leq n$ are polynomial identity rings satisfying a
polynomial identity of degree $\leq2n$. From this, it is clear that the
subdirect product $L$ also satisfies a polynomial identity of degree $\leq n$.
This proves (v).

The implication (v) $\implies$ (ii) has been established in \cite{BLR}.
\end{proof}

\noindent The Leavitt path algebra in Theorem \ref{bdd Index} need not
decompose as a direct sum of matrix rings, as the following example shows.

\begin{example}
\textrm{\label{InfiniteClock} Consider the following \textquotedblleft
infinite clock" graph $E$:%
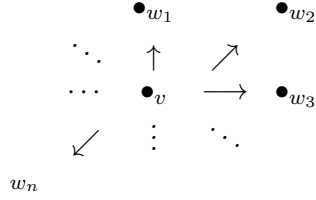
\[%
\begin{array}
[c]{ccccc}
&  & \bullet_{w_{1}} &  & \bullet_{w_{2}}\\
& \ddots & \uparrow & \nearrow & \\
& \cdots & \bullet_{v} & \longrightarrow & \bullet_{w_{3}}\\
& \swarrow & \vdots & \ddots & \\
_{w_{n}} &  &  &  &
\end{array}
\]
\noindent Thus $E^{0}=\{v\}\cup\{w_{1},w_{2},\cdot\cdot\cdot,w_{n},\cdot
\cdot\cdot\}$ where the $w_{i}$ are all sinks. For each $n\geq1$, let $e_{n}$
denote the single edge connecting $v$ to $w_{n}$. The graph $E$ is acyclic and
so every ideal of $L$ is graded (\cite{HR}). The number of distinct paths
ending at any given sink \ (including the sink) is $\leq2$. For each $n\geq1$,
$H_{n}=\{w_{i}:i\neq n\}$ is a hereditary saturated set, $B_{H_{n}}=\{v\}$ and
($E\backslash(H_{n},B_{H_{n}}))^{0}=\{v,w_{n}\}$ is downward directed. Also E%
$\backslash$%
(H$_{n}$,B$_{H_{n}}$) trivially satisfies Condition (L). Hence the ideal
$P_{n}$ generated by $H_{n}\cup\{v-e_{n}e_{n}^{\ast}\}$ is a\ graded primitive
ideal by Theorem 4.3(iii) of \cite{R} and $L_{K}(E)/P\cong M_{2}(K)$.
Moreover, every graded primitive (equivalently, prime) ideal $P$ of $L_{K}(E)$
is equal to $P_{n}$ for some $n$. By \cite[Theorem 4.12]{HRS-1}, $L_{K}(E)$ is
a graded $\Sigma$-$V$ ring. }

\textrm{But $L_{K}(E)$ cannot decompose as a direct sum of the matrix rings
$M_{2}(K) $. Because, otherwise, $v$ would lie in a direct sum of finitely
many copies of $M_{2}(K)$. Since the ideal generated by $v$ is $L_{K}(E)$,
$L_{K}(E)$ will then be a direct sum of finitely many copies of $M_{2}(K)$.
This is impossible since $L_{K}(E)$ contains an infinite set of orthogonal
idempotents $\{e_{n}e_{n}^{\ast}:n\geq1\}$. }

\textrm{We can also describe the internal structure of this ring $L_{K}(E)$.
The socle $S$ of $L_{K}(E)$ is the ideal generated by the sinks $\{w_{i}%
:i\geq1\}$, $S\cong\bigoplus\limits_{\aleph_{0}}M_{2}(K)$ and $L_{K}(E)/S\cong
K$. }
\end{example}

But the decomposition is possible if the graph is row-finite, as shown in the
following theorem.

\begin{theorem}
\label{row-finite bddIndex}Let $E$ be a row-finite graph. Then the following
properties are equivalent for $L:=L_{K}(E)$:

\begin{enumerate}
[(i)]

\item $L$ has bounded index of nilpotence $\leq n$;

\item There is a fixed positive integer $n$ and a graded isomorphism%

\[
L\cong_{gr}{\displaystyle\bigoplus\limits_{i\in I}}M_{n_{i}}(K)\oplus
{\displaystyle\bigoplus\limits_{j\in J}}M_{n_{j}}(K[x,x^{-1}])
\]
where $I,J$ are arbitrary index sets and, for all $i\in I$ and $j\in J$,
$n_{i}\leq n$ and $n_{j}\leq n$ . In particular, $L$ is graded semi-simple
(that is, a direct sum of graded simple left/right ideals).
\end{enumerate}
\end{theorem}

\begin{proof}
Assume (i). By Theorem \ref{bdd Index}, no cycle in $E$ has an exit and the
number of distinct paths that end at any vertex $v$ is $\leq n$, with the
proviso that if $v$ sits on a cycle $c$, then these paths do not include the
entire cycle $c$. If $A$ is the graded ideal generated by all the sinks in $E$
and all the vertices on cycles without exits, then, by Corollary 2.6.5 and
Lemma 2.7.1 of \cite{AArS}, $A\cong{\displaystyle\bigoplus\limits_{i\in I}%
}M_{n_{i}}(K)\oplus{\displaystyle\bigoplus\limits_{j\in J}}M_{n_{j}%
}(K[x,x^{-1}])$. By giving appropriate matrix gradings, this isomorphism
becomes a graded isomorphism. We claim that $L=A$. Let $H\subseteq A$ be the
set consisting of all the sinks and all the vertices on cycles in $E$. By
hypothesis, every path in $E$ that does not include an entire cycle has length
$\leq n$ and ends at a vertex in $H$. So if $u$ is any vertex in $E$, using
the fact that all the vertices in $E$ are regular and by a simple induction on
the length of the longest path from $u$, we can conclude that $u$ belongs to
the saturated closure of $H$. This implies that $L=A\cong_{gr}%
{\displaystyle\bigoplus\limits_{i\in I}}M_{n_{i}}(K)\oplus
{\displaystyle\bigoplus\limits_{j\in J}}M_{n_{j}}(K[x,x^{-1}])$. This proves (ii).

(ii) $\implies$ (i) follows from the fact that the matrix rings $M_{n_{i}}(K)
$ and $M_{n_{j}}(K[x,x^{-1}])$ with $n_{i},n_{j}\leq n$ have index of
nilpotence $\leq n$.
\end{proof}

One consequence of Theorem \ref{bdd Index} is the following.

\begin{prop}
\label{bdx => SigmaV} Let $E$ be an arbitrary graph. If $L:=L_{K}(E)$ has
bounded index of nilpotence $n$, then $L$ is a graded $\Sigma$-$V$ ring.
\end{prop}

\begin{proof}
If $L$ has bounded index of nilpotence $n$, then for any graded primitive
ideal $P$ of $L$ (since it is also graded prime), we have, by Theorem
\ref{bdd Index}(iii), $L/P\cong_{gr}M_{t}(K)$ or $M_{t}(K[x,x^{-1}])$ with
appropriate matrix gradings, where $t\leq n$. By \cite[Theorem 4.12]{HRS-1},
we then conclude that $L$ is a graded $\Sigma$-$V$ ring.
\end{proof}

The converse of the above result does not hold, as can be seen in the two
examples below.

\begin{example}
\textrm{\label{Inverse infinite clock} Let $E$ be the \textquotedblleft
inverse infinite clock" graph consisting of a sink $w$ and countably infinite
edges $\{e_{n}:n\geq1\}$ with $r(e_{n})=w$ and s(e$_{n}$) = w}$_{n}$ \textrm{
for all $n$. }

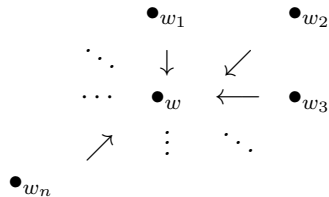
\textrm{%
\[%
\begin{array}
[c]{ccccc}
&  & \bullet_{w_{1}} &  & \bullet_{w_{2}}\\
& \ddots & \downarrow & \swarrow & \\
& \cdots & \bullet_{w} & \longleftarrow & \bullet_{w_{3}}\\
& \nearrow & \vdots & \ddots & \\
\bullet_{w_{n}} &  &  &  &
\end{array}
\]
} \textrm{\noindent Then $L_{K}(E)\cong M_{\infty}(K)$, the infinite
$\omega\times\omega$ matrix with finitely many non-zero entries. Now, under
appropriate matrix grading, $M_{\infty}(K)$ is graded semisimple (that is, a
graded direct sum of \ graded simple modules) and so all graded left/right
$M_{\infty}(K)$-modules are graded injective and hence $L$ is a graded
$\Sigma$-$V$ ring. But $L$ does not have bounded index of nilpotence by
Theorem \ref{bdd Index}(ii). }
\end{example}

\begin{example}
\textrm{\label{bdd idx no sigmaV} Consider the following graph $F$ consisting
of two cycles $g$ and $c$ connected by an edge:
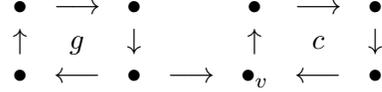
\[%
\begin{array}
[c]{ccccccc}%
\bullet & \longrightarrow & \bullet &  & \bullet & \longrightarrow & \bullet\\
\uparrow & g & \downarrow &  & \uparrow & c & \downarrow\\
\bullet & \longleftarrow & \bullet & \longrightarrow & \bullet_{v} &
\longleftarrow & \bullet
\end{array}
\]
} \textrm{\noindent Now $F$ is downward directed, $c$ is a cycle without exits
and the various powers of the cycle $g$ give rise to infinitely many distinct
paths that end at the base $v$ of the cycle $c$. Hence $L_{K}(F)\cong
M_{\infty}(K[x,x^{-1}])$ (by Lemma 2.7.1 of \cite{AArS}) and is graded
semisimple. Hence each graded simple module over $L_{K}(F)$ is graded $\Sigma
$-injective. But $L_{K}(F)$ does not have bounded index of nilpotence, as
$M_{\infty}(K[x,x^{-1}])$ contains subrings isomorphic to $M_{n}(K[x,x^{-1}])$
for every positive integer $n$. }
\end{example}

In the monograph \cite{JST}, the following open question (6.33: Problem 3,
Chapter 6) was raised:

\begin{question}
Does every exchange right $\Sigma$-$V$ ring have bounded index of nilpotence?
\end{question}

We answer this question in the negative in the following remark.

\begin{remark}
\textrm{Consider the graph $E$ of Example \ref{Inverse infinite clock}. As
noted there, $L_{K}(E)\cong M_{\infty}(K)$ $\mathrm{\ }$which is semisimple
and hence is a $\Sigma$-$V$ ring. Since $E$ is acyclic, $L_{K}(E)$ is von
Neumann regular \cite{AR} and hence is an exchange ring. But $L_{K}(E)$ does
not have bounded index of nilpotence by Theorem \ref{bdd Index}(ii). }
\end{remark}

\begin{remark}
\textrm{It was shown in \cite{HRS-1} that a Leavitt path algebra $L_{K}(E)$ is
directly-finite (equivalently, graded directly-finite with respect to vertices) if and only if no cycle in $E$ has an exit. In view of Theorem
\ref{bdd Index}, it is clear that a Leavitt path algebra of bounded index is
always directly-finite. But, for arbitrary rings with bounded index of
nilpotence, this is not the case. Let $S=\prod R_{k}$, where each $R_{k}%
\cong\mathbb{Z}(p^{n})$, the ring of integers modulo a fixed integer $n\geq2$.
Now $(pS)^{n}=0$ and if $a\in S$ is nilpotent, then $a\in pS$ and $a^{n}=0$.
Consequently, $S$ has bounded index of nilpotence. In fact, the index of
nilpotence of $S$ is $n$. But $S$ is not directly-finite, since $S\cong%
\prod\limits_{k\geq2}R_{k}\cong\prod\limits_{k\geq3}R_{k}\cong\cdot\cdot\cdot
$. }

\textrm{Conversely, a directly-finite Leavitt path algebra need not have
bounded index of nilpotence. If $E$ is the graph consisting of an infinite
line segment
\[
\bullet\longrightarrow\bullet\longrightarrow\bullet\longrightarrow\cdot
\cdot\cdot\bullet\longrightarrow\cdot\cdot\cdot
\]
then clearly $L_{K}(E)$ is directly-finite, but, by Theorem \ref{bdd Index}
(ii), $L_{K}(E)\cong M_{\infty}(K)$ does not have bounded index of nilpotence.
}
\end{remark}

\end{document}